\documentclass[11pt]{article}
\usepackage{amsmath,amsthm}
\usepackage{amssymb,mathrsfs}
\usepackage{bm,mathtools}
\usepackage{graphics}
\usepackage{tikz}
\usepackage{float}
\usepackage{enumitem}
\usepackage{bm}
\usetikzlibrary{calc}

\usepackage{xcolor}

\usepackage{geometry}
\usepackage[colorlinks=true,
linkcolor=blue,citecolor=blue,pagebackref,
urlcolor=blue]{hyperref}

\makeatletter
\def\@seccntDot{.}
\def\@seccntformat#1{\csname the#1\endcsname\@seccntDot\hskip 0.5em}
\renewcommand\section{\@startsection{section}{1}{\z@}%
	{18\p@ \@plus 6\p@ \@minus 3\p@}%
	{9\p@ \@plus 6\p@ \@minus 3\p@}%
	{\large\bfseries\boldmath}}
\renewcommand\subsection{\@startsection{subsection}{2}{\z@}%
	{12\p@ \@plus 6\p@ \@minus 3\p@}%
	{3\p@ \@plus 6\p@ \@minus 3\p@}%
	{\bfseries\boldmath}}
\renewcommand\subsubsection{\@startsection{subsubsection}{3}{\z@}%
	{12\p@ \@plus 6\p@ \@minus 3\p@}%
	{\p@}%
	{\bfseries\boldmath}}
\makeatother

\usepackage{microtype}
\usepackage{float}

\newcommand{\keywords}[1]{%
  \par\medskip
  \noindent{\small\textbf{Keywords:} #1\par}
}

\newcommand{\msc}[1]{%
  \par\smallskip
  \noindent{\small\textbf{MSC:} #1\par}
  \medskip
}

\theoremstyle{plain}
\newtheorem{theorem}{Theorem}[section]
\newtheorem{lemma}[theorem]{Lemma}

\newtheorem{conj}[theorem]{Conjecture}

\theoremstyle{definition}
\newtheorem{definition}[theorem]{Definition}

\newtheorem{claim}{Claim}

\allowdisplaybreaks
\DeclareMathOperator{\tr}{tr}

\newcommand{\ii}{\mathrm{i}}
\newcommand{\D}{\Delta}
\newcommand{\eps}{\varepsilon}

\title{A Proof of a Conjecture on Positive and Negative Square Energies of Unicyclic Graphs}

\author{
	Bo Ning\thanks{College of Computer Science \& College of Cryptology and Cyber Science, Nankai University,
		Tianjin 300350, P.R. China. E-mail: \texttt{bo.ning@nankai.edu.cn}.}
	\and 
    Jing Zeng\thanks{College of Cryptology and Cyber Science, Nankai University, Tianjin 300350, P.R. China. E-mail: \texttt{jingzeng@mail.nankai.edu.cn}.}
}

\date{}

\begin{document}
	\maketitle		
	
	\begin{abstract}
		Let $G$ be a unicyclic graph of order $n$, and let $k$ be the length of the unique cycle of $G$. For the adjacency eigenvalues of $G$, let $s^{+}(G)$ and $s^{-}(G)$ denote the sums of the squares of the positive and negative eigenvalues, respectively. Akbari, Kumar, Mohar, Pragada, and Zhang conjectured that, when $k$ is odd, the value of $k$ modulo $4$ determines which of $s^+(G)$ and $s^-(G)$ is greater than $n$. More precisely, if $k\equiv 3\pmod 4$, then $s^+(G)>n>s^-(G)$; if $k\equiv 1\pmod 4$, then $s^+(G)<n<s^-(G)$. We confirm this conjecture.
	\end{abstract}

    \msc{05C50, 05C31}

\section{Introduction}\label{sec-intro}

Let $G$ be a finite simple graph of order $n$, and let
$\lambda_1(G)\ge \lambda_2(G)\ge \cdots \ge \lambda_n(G)$
be the eigenvalues of its adjacency matrix. The positive and negative square energies of $G$ are
\[
 s^+(G)=\sum_{\lambda_i(G)>0}\lambda_i(G)^2,
 \qquad
 s^-(G)=\sum_{\lambda_i(G)<0}\lambda_i(G)^2.
\]
They arise naturally in problems concerning spectral lower bounds for graph parameters, refinements of graph energy, and Schatten-type norms; see, for instance,
\cite{AbiadLDGHM23,ElphickFGW16,ElphickLinz24,Nikiforov16,WocjanElphick13,Zhang24}.

Substituting the spectral radius by $s^+(G)$ has successfully extended and strengthened several classical spectral inequalities. For example, improving Stanley's bound $\lambda_1(G)\leq -\frac{1}{2}+\sqrt{2m+\frac{1}{4}}$ \cite{Stanley87}, Wu and Elphick \cite{WuElphick17} proved $\sqrt{s^+}\leq \frac{1}{2}(\sqrt{8m+1}-1)$. Wocjan and Elphick \cite{WocjanElphick13} conjectured a variant of Hoffman's inequality \cite{Hoffman1970}, namely $\chi(G)\geq 1+\max\{\frac{s^+}{s^-},\frac{s^-}{s^+}\}$, which was proved by Ando and Lin \cite{AndoLin15}. Hong's well-known inequality \cite{Hong1988} states that any connected graph with $n$ vertices and $m$ edges satisfies $\lambda_1(G)\leq \sqrt{2m-n+1}$. Extending this, Elphick, Farber, Goldberg, and Wocjan \cite{ElphickFGW16} conjectured that every connected graph on $n$ vertices satisfies $\min\{s^+(G),s^-(G)\}\ge n-1$.
This conjecture is known for several graph classes, including cycles, regular graphs, and various dense or
structured families; see \cite{AbiadLDGHM23,ElphickFGW16,Zhang24}. The general lower bounds have
also been improved recently. Elphick and Linz \cite{ElphickLinz24} studied symmetry and asymmetry
between $s^+(G)$ and $s^-(G)$ and obtained a general lower bound of order $\sqrt n$, which is related to
a chromatic-number result of Ando and Lin \cite{AndoLin15}. More recently, Akbari, Kumar, Mohar and
Pragada \cite{AkbariKMP2025} proved that
for any connected graph $G$ on $n \geq 4$ vertices,
$\min\{s^+(G),s^-(G)\}\geq \frac{3n}{4}$, which is the first linear lower bound for the square energy.
Their proof relies on the following super-additivity result established in  \cite{AkbariKMP2025}:
if $H_1,\ldots,H_r$ are pairwise disjoint induced subgraphs of a graph $G$, then $s^+(G)\ge \sum_{j=1}^r s^+(H_j)$ and $s^-(G)\ge \sum_{j=1}^r s^-(H_j)$. Moreover, equality holds in both inequalities simultaneously if and only if $G$ is the
disjoint union of $H_1,\ldots,H_r$.

Further refinements were proposed and studied by Akbari, Kumar, Mohar, Pragada, and Zhang
\cite{AkbariKMPZ25}. In particular, they suggested the stronger inequality $s^+(G)\ge n$ for connected
graphs on $n$ vertices and at least $n+1$ edges. For this conjecture, they verified it for claw-free graphs and graphs of diameter $2$,
and proved positive square-energy lower bounds for graphs with small domination number.

It is known that 
a unicyclic
graph is a connected graph on $n$ vertices and $n$ edges and it contains exactly one cycle.
Akbari et al. \cite{AkbariKMPZ25} computed the square energies of cycles
exactly. For an odd cycle $C_k$, their result gives
\[
 s^+(C_k)-s^-(C_k)=
 \begin{cases}
 2\bigl(\sec(\frac{\pi}{k})-1\bigr),& k\equiv 3\pmod 4;\\[2mm]
 -2\bigl(\sec(\frac{\pi}{k})-1\bigr),& k\equiv 1\pmod 4.
 \end{cases}
\]
In \cite{AkbariKMPZ25},
they proposed the following conjecture, which concerns all unicyclic graphs with an odd cycle.

\begin{conj}[Akbari--Kumar--Mohar--Pragada--Zhang {\cite{AkbariKMPZ25}}]\label{conjunicyc}
If $G$ is a unicyclic graph of order $n$ whose cycle has odd length $k$, then
\begin{enumerate}[label=(\roman*)]
    \item $s^+(G)>n>s^-(G)$, if $k\equiv 3 \pmod 4$.
    \item $s^+(G)<n<s^-(G)$, if $k\equiv 1 \pmod 4$.
\end{enumerate}
\end{conj}

In this paper, we resolve Conjecture~\ref{conjunicyc}.

\begin{theorem}\label{mainthm}
Let $G$ be a unicyclic graph of order $n$, and let the length of the unique cycle of $G$ be $k$.
\begin{enumerate}[label=(\roman*)]
\item If $k\equiv 0\pmod 2$, then $s^{+}(G)=s^{-}(G)=n$.

\item If $k\equiv 3\pmod 4$, then $s^+(G)>n>s^-(G)$.
\item If $k\equiv 1\pmod 4$, then $s^+(G)<n<s^-(G)$.
\end{enumerate}
\end{theorem}

Our proof is based on two facts. First, the characteristic polynomial of a unicyclic graph differs from its matching polynomial by exactly one term, contributed by the unique cycle. Second, a Coulson-type integral expresses $s^+(G)-s^-(G)$ through the argument of the
characteristic polynomial on the imaginary axis. For an odd cycle, the correction term in the
unicyclic Sachs expansion is purely imaginary after the normalization
$\varphi_G(\ii t)\mapsto \ii^{-n}\varphi_G(\ii t)$, and its sign is determined by $k\pmod 4$. It is known that Coulson-type integral is closely related to energy problems, see for example \cite{QiaoZhangNingLi16,QiaoZhangLi17,QiaoZhangLiGao23}.

Throughout this paper, all graphs are simple, finite, and undirected. We denote the adjacency matrix of a graph $G$ by $A(G)$,
and we use the convention $\phi_G(x)=\det(xI-A(G))$ for the characteristic polynomial of $G$.

This paper is organized as follows. In Section \ref{sec-intro}, we give an introduction to our topic. In Section \ref{sec-proof}, we confirm Conjecture \ref{conjunicyc}.

\section{\texorpdfstring{The proof of Theorem~\ref{mainthm}}{The proof of Theorem 1.2}}\label{sec-proof}

We use the following standard normalization of the matching polynomial, see, e.g., Godsil and Gutman~\cite{GodsilGutman81}.

\begin{definition}\label{defmatchpoly}
Let $H$ be a graph of order $n$, and $m_j(H)$ be the number of matchings of size $j$ in $H$. The matching polynomial of $H$ is
\[
 \mu_H(x)=\begin{cases}
    \displaystyle \sum_{j=0}^{\lfloor \frac{n}{2} \rfloor}(-1)^j m_j(H)x^{n-2j}, & \text{if $n \neq 0$};\\
    1,& \text{if $n=0$}.
 \end{cases}
\]
\end{definition}

Before giving the proof of Theorem \ref{mainthm}, we need the following results. We first recall the Sachs expansion for the characteristic polynomial.

\begin{theorem}[Sachs expansion {\cite{Sachs1964}}]\label{sachsexpan}
Let $G$ be a graph of order $n$, and let $\mathcal{E}(G)$ be the set of all elementary subgraphs of $G$, where an
elementary subgraph is a subgraph whose connected components are either single edges or cycles. The empty subgraph is included in $\mathcal{E}(G)$.
For $H\in\mathcal{E}(G)$, let $c(H)$ be the number of connected components of
$H$, $c_{\circ}(H)$ be the number of cycle components of $H$, and let
$v(H)=|V(H)|$. Then
\begin{equation}\label{eqSachspoly}
    \phi_G(x)
=
\sum_{H\in\mathcal{E}(G)}
(-1)^{c(H)}2^{c_{\circ}(H)}x^{n-v(H)}.
\end{equation}
\end{theorem}

Applying Theorem \ref{sachsexpan} to unicyclic graphs gives the following special case.

\begin{lemma}\label{lemunicycpoly}
Let $G$ be a unicyclic graph, $C$ be its unique cycle, and let $F=G-V(C)$.
Then
$\phi_G(x)=\mu_G(x)-2\mu_F(x)$.
\end{lemma}

\begin{proof}
Let $k=|V(C)|$ and $n=|V(G)|$. By Theorem \ref{sachsexpan}, the characteristic
polynomial $\phi_G(x)$ is obtained by summing over all elementary subgraphs of
$G$.

First we consider those elementary subgraphs whose components are only consisting of single edges. These are precisely the matchings of $G$. A matching of size $j$ has
$j$ connected components, no cycle components, and covers $2j$ vertices. Hence
its contribution in \eqref{eqSachspoly} is $(-1)^j x^{n-2j}$.
Summing over all matchings of $G$ gives 
\begin{equation}\label{equnicase1}
\sum_{j=0}^{\lfloor \frac{n}{2} \rfloor} (-1)^{j}m_j(G)x^{n-2j}=\mu_G(x).    
\end{equation}
It remains to consider elementary subgraphs containing a cycle component. Since
$G$ is unicyclic, the only possible cycle component is the unique cycle $C$, which implies that all remaining components must be single edges disjoint from
$C$, and therefore they form a matching in $F=G-V(C)$. If this matching has size $j$, then the corresponding elementary subgraph has
$j+1$ connected components in which there is one cycle component, and covers $k+2j$ vertices.
Hence its contribution in \eqref{eqSachspoly} is $(-1)^{j+1}2x^{n-k-2j}=-2(-1)^jx^{|V(F)|-2j}$. Summing over these elementary subgraphs containing a cycle component gives
\begin{equation}\label{equnicase2}
-2\sum_{j=0}^{\lfloor \frac{n-k}{2} \rfloor}(-1)^jm_j(F)x^{|V(F)|-2j}=-2\mu_F(x).    
\end{equation}

All elementary subgraphs of $G$ are contained in the above two cases. Combining \eqref{equnicase1} and \eqref{equnicase2}, we have $\phi_G(x)=\mu_G(x)-2\mu_F(x)$. This completes the proof.
\end{proof}

Motivated by the classical Coulson integral formula for graph energy
\cite{Coulson40}, we prove the following signed square-energy
analog.

\begin{lemma}[Coulson-type identity]\label{lemcoulson}
Let $G$ be a graph of order $n$ with eigenvalues $\lambda_1\geq \cdots \geq \lambda_n$ of $A(G)$, and let
$\Theta_G(t)=\sum_{i=1}^n\arctan\left(\frac{\lambda_i}{t}\right)$
for $t>0$, where $\arctan\left(\frac{\lambda_i}{t}\right)$ takes values in $(-\frac{\pi}{2},\frac{\pi}{2})$. Then
\begin{equation*}
    s^+(G)-s^-(G)=-\frac{4}{\pi}\int_0^\infty t\Theta_G(t)\,dt.
\end{equation*}
\end{lemma}

\begin{proof}
Let $\D(G)=s^+(G)-s^-(G)=\sum_{i=1}^n\lambda_i|\lambda_i|$. We first show that, for each $\lambda_i$,
\begin{equation}\label{eqlambdaab}
    \lambda_i|\lambda_i|=\frac{2}{\pi}\int_0^\infty\frac{\lambda_i^3}{\lambda_i^2+t^2}\,dt.
\end{equation}
In fact, if $\lambda_i= 0$, it is easy to see that \eqref{eqlambdaab} holds; if $\lambda_i\neq 0$, since $\int_0^\infty\frac{dt}{\lambda^2+t^2}=\frac{\pi}{2|\lambda|}$, multiplying both sides by $\lambda^{3}_i$ yields \eqref{eqlambdaab}.

Moreover, for each $\lambda_i$, the integral in \eqref{eqlambdaab} is absolutely convergent since $\int_0^\infty\left|\frac{\lambda^3}{\lambda^2+t^2}\right|\,dt<\infty$.

By \eqref{eqlambdaab} and the fact that $n$ is finite, we have
\begin{equation}\label{eqdeltaG}
 \D(G)=
 \sum_{i=1}^n\frac{2}{\pi}\int_0^\infty\frac{\lambda_i^3}{\lambda_i^2+t^2}\,dt=\frac{2}{\pi}\int_0^\infty
 \sum_{i=1}^n\frac{\lambda_i^3}{\lambda_i^2+t^2}\,dt.
\end{equation}
Since $\sum_i\lambda_i=\tr A(G)=0$, we obtain
\begin{equation}\label{eqsumlambda}
 \sum_{i=1}^n\frac{\lambda_i^3}{\lambda_i^2+t^2}
 =\sum_{i=1}^n\left(\lambda_i-\frac{\lambda_i t^2}{\lambda_i^2+t^2}\right)
 =-t^2\sum_{i=1}^n\frac{\lambda_i}{\lambda_i^2+t^2}.
\end{equation}
On the other hand,
\begin{equation}\label{eqthetade}
 \Theta_G'(t)=\sum_{i=1}^n\frac{d\left(\arctan\!\left(\frac{\lambda_i}{t}\right)\right)}{dt}
 =-\sum_{i=1}^n\frac{\lambda_i}{\lambda_i^2+t^2}.
\end{equation}
Combining \eqref{eqsumlambda} and \eqref{eqthetade} gives
\begin{equation}\label{eqlambdatheta}
 \sum_{i=1}^n\frac{\lambda_i^3}{\lambda_i^2+t^2}=t^2\Theta_G'(t).
\end{equation}
It follows from \eqref{eqlambdatheta} that
$ \frac{2}{\pi}\int_\eps^R
 \sum_{i=1}^n\frac{\lambda_i^3}{\lambda_i^2+t^2}\,dt
 =\frac{2}{\pi}\int_\eps^R t^2\Theta_G'(t)\,dt$ for $0<\eps<R$.
Integrating by parts gives
\begin{equation}\label{eqparts}
 \frac{2}{\pi}\int_\eps^R t^2\Theta_G'(t)\,dt
 =\frac{2}{\pi}\bigl[t^2\Theta_G(t)\bigr]_{\eps}^{R}
 -\frac{4}{\pi}\int_\eps^R t\Theta_G(t)\,dt.
\end{equation}

It remains to verify that $\eps^2\Theta_G(\eps)\to 0$ as $\eps\to 0^{+}$ and $R^2\Theta_G(R)\to 0$ as $R\to\infty$, and that $\int_0^\infty t\Theta_G(t)\,dt$ converges.

Since $\left|\arctan(\frac{\lambda_i}{t})\right|\le \frac{\pi}{2}$ for every $t>0$, we have $|\Theta_G(t)|\le \frac{n\pi}{2}$. Hence $|\eps^2\Theta_G(\eps)|\le \frac{n\pi}{2}\eps^2\to 0$ as $\eps\to 0^{+}$. On the other hand, as $t\to\infty$, we have $\arctan\left(\frac{\lambda_i}{t}\right)=\frac{\lambda_i}{t}+O(t^{-3})$. Since $\sum_i\lambda_i=0$, it follows that $\Theta_G(t)=O(t^{-3})$. Thus $R^2\Theta_G(R)=O(R^{-1})\to 0$ as $R\to\infty$.

Similarly, using $|\Theta_G(t)|\le \frac{n\pi}{2}$ gives $|t\Theta_G(t)|\le \frac{n\pi}{2}t$,
and thus $ \int_0^1 |t\Theta_G(t)|\,dt
 \le \frac{n\pi}{2}\int_0^1 t\,dt<\infty$.
Moreover, using $\Theta_G(t)=O(t^{-3})$ gives $R^2\Theta_G(R)=O(R^{-1})$, and thus there exist constants $C>0$ and $T>1$ such that $|t\Theta_G(t)|\le Ct^{-2}$ for all $t\ge T$. Consequently,
$ \int_T^\infty |t\Theta_G(t)|\,dt
 \le C\int_T^\infty t^{-2}\,dt<\infty$.
Since $\Theta_G(t)$ is a finite sum of continuous functions on $(0,\infty)$,
$t\Theta_G(t)$ is continuous on every compact subinterval of $(0,\infty)$. Thus $ \int_1^T |t\Theta_G(t)|\,dt<\infty$ for every $T>1$.
By these arguments, $\int_0^\infty t\Theta_G(t)\,dt$ is absolutely convergent, and hence convergent.

Letting $\eps\to 0^{+}$ and $R\to\infty$ in \eqref{eqparts} and using
\eqref{eqdeltaG} together with \eqref{eqlambdatheta}, we obtain $ \D(G)=-\frac{4}{\pi}\int_0^\infty t\Theta_G(t)\,dt$. This completes the proof.
\end{proof}

\begin{proof}[\textbf{Proof of Theorem \ref{mainthm}}]
Since $G$ is unicyclic, we have
\begin{equation}\label{eqsumsquare0}
 s^+(G)+s^-(G)
 =\sum_{i=1}^n\lambda_i(G)^2
 =\tr A(G)^2
 =2|E(G)|
 =2n.
\end{equation}
Let $\D(G)=s^+(G)-s^-(G)$.
It is enough to determine the sign of $\D(G)$. Let $C$ be the unique cycle of $G$, and let $F=G-V(C)$. For a graph $H$ and $t>0$, define
\begin{equation}\label{eqdefMH}
   M_H(t)=\ii^{-v(H)}\mu_H(\ii t),
\end{equation}
where $\ii$ denotes the imaginary unit, and $\mu_H$ denotes the matching polynomial of $H$ defined in
Definition \ref{defmatchpoly}. To finish the proof, we establish the following claims.
\begin{claim}\label{claimmatch}
For every graph $H$ and every $t>0$, $M_H(t)=\sum_{j\ge 0}^{\lfloor \frac{v(H)}{2} \rfloor}m_j(H)t^{v(H)-2j}>0$, where $v(H)=|V(H)|$.
Moreover, $M_H(t)=t^{v(H)}\bigl(1+O(t^{-2})\bigr)$ as $t\to\infty$. 
\end{claim}

\begin{proof}[\textbf{Proof of Claim~\ref{claimmatch}}]
By Definition \ref{defmatchpoly}, we have
\[\mu_H(\ii t)=\sum_{j\ge 0}^{\lfloor \frac{v(H)}{2} \rfloor}(-1)^jm_j(H)(\ii t)^{v(H)-2j}=\sum_{j\ge 0}^{\lfloor \frac{v(H)}{2} \rfloor}\ii^{v(H)}m_j(H)t^{v(H)-2j}.\]
Combining this with \eqref{eqdefMH}, we obtain 
\begin{equation}\label{eqMHsum}
    M_H(t)=\sum_{j\ge 0}^{\lfloor \frac{v(H)}{2} \rfloor}m_j(H)t^{v(H)-2j}.
\end{equation}
On the right-hand side of \eqref{eqMHsum}, all terms are nonnegative for $t>0$, thus the sum is positive. Moreover, the term corresponding to $j=0$ is $t^{v(H)}$. It follows from \eqref{eqMHsum} that $M_H(t)=t^{v(H)}\bigl(1+O(t^{-2})\bigr)$ as $t\to\infty$. This completes the proof of Claim \ref{claimmatch}.
\end{proof}

\begin{claim}\label{claimargu}
Let $\Theta_G(t)$ be defined as in Lemma \ref{lemcoulson}. Then
\begin{equation*}
    \Theta_G(t)=
 \begin{cases}
 \displaystyle \arctan\left(\frac{2M_F(t)}{M_G(t)}\right), & k\equiv 1\pmod 4;\\[3mm]
 \displaystyle -\arctan\left(\frac{2M_F(t)}{M_G(t)}\right), & k\equiv 3\pmod 4.
 \end{cases}
\end{equation*}
Consequently, $\Theta_G(t)>0$ for all $t>0$ when $k\equiv 1\pmod 4$, and $\Theta_G(t)<0$ for all $t>0$ when $k\equiv 3\pmod 4$.
\end{claim}

\begin{proof}[\textbf{Proof of Claim \ref{claimargu}}]
By the definition of $\Theta_G(t)$, for $t>0$,
\begin{equation}\label{eqdefThetaG}
\Theta_G(t)=\sum_{i=1}^n\arctan\left(\frac{\lambda_i}{t}\right).
\end{equation}
By Lemma \ref{lemunicycpoly},
$\phi_G(\ii t)=\mu_G(\ii t)-2\mu_F(\ii t).$
Using Claim \ref{claimmatch}, we have
$\phi_G(\ii t)=\ii^nM_G(t)-2\ii^{n-k}M_F(t),$
and hence
\begin{equation*}
 \ii^{-n}\phi_G(\ii t)=M_G(t)-2\ii^{-k}M_F(t).
\end{equation*}
If $k\equiv 1\pmod 4$, then $\ii^{-k}=-\ii$, and thus
$\ii^{-n}\phi_G(\ii t)=M_G(t)+2\ii M_F(t)$.
If $k\equiv 3\pmod 4$, then $\ii^{-k}=\ii$, and thus
$\ii^{-n}\phi_G(\ii t)=M_G(t)-2\ii M_F(t)$.
In both cases, Claim~\ref{claimmatch} gives $M_G(t)>0$ and $M_F(t)>0$. It follows that $\ii^{-n}\phi_G(\ii t)$ has positive real part for every $t>0$. Let $\Psi_{G}(t)$ denote the principal argument of $\ii^{-n}\phi_G(\ii t)$. Then
\begin{equation*}
    \Psi_G(t)=
 \begin{cases}
 \displaystyle \arctan\!\left(\frac{2M_F(t)}{M_G(t)}\right), & k\equiv 1\pmod 4;\\[3mm]
 \displaystyle -\arctan\!\left(\frac{2M_F(t)}{M_G(t)}\right), & k\equiv 3\pmod 4.
 \end{cases}
\end{equation*}

Next we show $\Theta_{G}(t)=\Psi_{G}(t)$. A straightforward calculation gives $\ii^{-n}\phi_G(\ii t)=\prod_{i=1}^n(t+\ii\lambda_i)$. Since $t+\ii\lambda_i$ has positive real part, the principal argument of each
factor $t+\ii\lambda_i$ is
\[
 \operatorname{Arg}(t+\ii\lambda_i)
 =\arctan\!\left(\frac{\lambda_i}{t}\right)
 \in \left(-\frac{\pi}{2},\frac{\pi}{2}\right).
\]
Therefore, by \eqref{eqdefThetaG}, $\Theta_G(t)$ is a continuous argument of
$\ii^{-n}\phi_G(\ii t)$. Now we have $\Theta_G(t)-\Psi_{G}(t)=2q\pi$ for every $t>0$, where $q \in \mathbb{Z}$ is a constant. Since both $\Theta_G(t)$ and $\Psi_{G}(t)$ are continuous on $(0,\infty)$, the
function $\Theta_G(t)-\Psi_{G}(t)$ is continuous. Moreover, $\Theta_G(t)\to 0$ as $t\to\infty$ by \eqref{eqdefThetaG}. Also, Claim \ref{claimmatch} gives
\[
 \frac{2M_F(t)}{M_G(t)}=O(t^{-k}),
\]
and hence $\Psi_{G}(t)\to 0$ as $t\to\infty$. It follows that the constant
$q$ is equal to $0$. Thus $\Theta_G(t)=\Psi_{G}(t)$ for all $t>0$.

Finally, since $M_G(t)>0$ and $M_F(t)>0$ for every $t>0$, we have
\[
 \arctan\!\left(\frac{2M_F(t)}{M_G(t)}\right)>0.
\]
Thus $\Theta_G(t)>0$ for all $t>0$ when $k\equiv 1\pmod 4$, and
$\Theta_G(t)<0$ for all $t>0$ when $k\equiv 3\pmod 4$. This completes the proof of the claim.
\end{proof}

By Lemma \ref{lemcoulson},
\begin{equation}\label{eqdelta2}
 \D(G)=-\frac{4}{\pi}\int_0^\infty t\Theta_G(t)\,dt.
\end{equation}

If $k\equiv 0\pmod 2$, then $G$ is a bipartite graph since $G$ contains no cycles with odd length. Hence $s^{+}(G)=s^{-}(G)$. Combining this with \eqref{eqsumsquare0} gives $s^{+}(G)=s^{-}(G)=n$.

If $k\equiv 1\pmod 4$, then Claim \ref{claimargu} gives $\Theta_G(t)>0$ for all $t>0$. Hence \eqref{eqdelta2} yields $\D(G)<0$. Combining this with \eqref{eqsumsquare0} gives $s^+(G)=n+\frac{\D(G)}2<n$ and $s^-(G)=n-\frac{\D(G)}2>n$.

If $k\equiv 3\pmod 4$, then Claim \ref{claimargu} gives $\Theta_G(t)<0$ for all $t>0$. Hence \eqref{eqdelta2} yields $\D(G)>0$. Combining this with \eqref{eqsumsquare0} gives $s^+(G)=n+\frac{\D(G)}2>n$ and $s^-(G)=n-\frac{\D(G)}2<n$.

This completes the proof.
\end{proof}

	
\bibliographystyle{unsrt}
 
\end{document}